\documentclass[11pt]{amsart}
\usepackage{amsfonts}
\usepackage{amssymb}
\usepackage{amsmath} \numberwithin{equation}{section}
\usepackage{dsfont}
\usepackage{color}
\usepackage{graphicx}
\newtheorem{theorem}{Theorem}[section]
\newtheorem{lemma}[theorem]{Lemma}
\newtheorem{proposition}[theorem]{Proposition}
\newtheorem{corollary}[theorem]{Corollary}

\newtheorem{conjecture}[theorem]{Conjecture}

\theoremstyle{definition}
\newtheorem{df}{Definition}
\newtheorem{example}[df]{Example}

\newcommand{\N}{\mathbb N}

\newcommand{\R}{\mathbb R}

\newcommand{\ha}{\symbol{94}}
\newcommand{\ve}{\varepsilon}
\newcommand{\on}{\operatorname}

\newcommand{\wh}{\widehat}

\makeatletter
\@namedef{subjclassname@2020}{%
  \textup{2020} Mathematics Subject Classification}
\makeatother


\begin{document}
\author{Piotr Nowakowski}
\address{Faculty of Mathematics and Computer Science, University of \L \'{o}d\'{z},
Banacha 22, 90-238 \L \'{o}d\'{z}, Poland}
\address{Institute of Mathematics, Czech Academy of Sciences,
\v{Z}itn\'a 25, 115 67 Prague 1, Czech Republic\\
ORCID: 0000-0002-3655-4991}
\email{piotr.nowakowski@wmii.uni.lodz.pl}

\title{When the algebraic difference of two central Cantor sets is an interval?}
\date{}

\begin{abstract}
Let $C(a ),C(b)\subset \lbrack 0,1]$ be the central Cantor sets
generated by sequences $a,b \in \left( 0,1%
\right) ^{\mathbb{N}}$. The first main
result of the paper gives a necessary and a sufficient condition for sequences $a$ and $b$ which inform when $C(a )-C(b)$ is equal to $[-1,1]$ or is a finite union of closed intervals. One of the corollaries following from this results shows that the product of thicknesses of two central Cantor sets which algebraic difference is an interval may be arbitrarily small. We also show that there are sets $C(a)$ and $C(b)$ with the Hausdorff dimension equal to $0$ such that their algebraic difference is an interval. 
Finally, we give a full characterization of the case, when $C(a )-C(b)$ is equal to $[-1,1]$ or is a finite union of closed intervals. 
\end{abstract}

\maketitle

\section{Introduction}
For $A,B\subset \mathbb{R}$, we denote by $A\pm B$ the set $\left\{ a\pm
b:a\in A,\,b\in B\right\} $. The set $A- B$ is called the algebraic difference of sets $A$ and $B$. The set $A-A$ is called the difference set of a set $A$. We will also write $a+A$ instead $\{a\} + A$ for $a \in \R$. 
If $I \subset \R$ is an interval, then by $l(I)$, $r(I)$ we will denote, respectively, the left and the right endpoint of $I$.

By a Cantor set we mean a nonempty, bounded, perfect and nowhere dense subset of $\R$.

Given any set $C \subset \R$, every bounded component of 
the set $\R \setminus C$ is called a gap of $C$. A component of $C$ is called proper if it is not a singleton.

Let us recall the definition of a Cantorval (specifically, an M-Cantorval). A perfect set $E \subset \R$ is called a Cantorval if it has infinitely many gaps and both endpoints of any gap are accumulated by gaps and proper components of $E$.

Algebraic differences and sums of Cantor sets were considered by many authors (e.g. \cite{MO}, \cite{T19}, \cite{K1}, \cite{HKY}, \cite{N}, \cite{H}, \cite{AC}, \cite{S}, \cite{FF}). They appear for example in dynamical systems (see \cite{PT}), spectral theory (see \cite{DGS},\cite{T16}) or number theory (see \cite{H}). One of the most known results is the Newhouse gap lemma from \cite{N} (see Theorem \ref{g1}). This is a condition which implies that the sum of two Cantor sets is an interval.
Also other mathematical operations on Cantor sets are in the area of interests of many mathematicians (see e.g. \cite{ART}, \cite{J}). The most popular types of examined Cantor sets are central Cantor sets. In our paper we will continue this trend.  

Every central Cantor subset $C$ of $[0,1]$ can be uniquely described by a sequence $a=(a_n)\in (0,1)^\N$ (the details are given
in Section 2). We then say that $C$ is generated by $a$ and write $C=C(a)$. If $a$ is a constant sequence with all terms equal to $\alpha$, then we say that the set $C(a)$ is a middle-$\alpha$ Cantor set or just a middle Cantor if $\alpha$ is not given. The algebraic difference of two central Cantor sets can be either a Cantor set, a finite union of closed intervals or a Cantorval
(see \cite{AI}).
In the paper \cite{K}, Kraft proved that the difference set of a middle-$\alpha$ Cantor set is equal to $[-1,1]$ if $\alpha \leq \frac{1}{3}$, and is a Cantor set if $\alpha > \frac{1}{3}$.
Later, it was proved (see \cite{AI}, \cite{FN}) that the difference set of a central Cantor set $C(a)$, where $a \in (0,1)^{\N}$, is equal to $[-1,1]$ if and only if $a_n \leq \frac{1}{3}$ for all $n \in \N$. In \cite{S}, the author proved that if $a_n > \frac{1}{3}$ for all $n \in \N$, then the difference set of $C(a)$ is a Cantor set. In \cite{FN} there was given a condition, which implies that the set $C(a)-C(a)$ is a Cantorval. In our paper, we will focus only on the case, when the algebraic difference of two central Cantor sets is a finite union of closed intervals. Although the characterization of this case is known for the difference $C(a)-C(a)$, there was not such characterization for the sets of the form $C(a)-C(b)$ for arbitrary sequences $a, b \in (0,1)^\N$. There are only some partial results like Newhouse gap lemma or theorem of Pourbarat from \cite{P}, which gives an equivalent condition for the algebraic difference of middle Cantor sets to be an interval (see also Corollary \ref{WnP}). There are also some results for more general Cantor sets. In \cite{LFGZJ} there are considered continuous images of pairs of Moran sets. Results obtained there are very interesting, but if we use them for the case of the algebraic difference of central Cantor sets, we will only obtain the known characterization for the set $C(a)-C(a)$ to be an interval. 
The main goal of this paper is to give a characterization of the case, where the set $C(a)-C(b)$ is a finite union of intervals. 

In Section 2, we start from proving a weaker theorem which concerns the algebraic difference of two different central Cantor sets, but which is easier to apply than the characterization given later. It is a generalization of the earlier mentioned theorem concerning the difference set of a central Cantor set. It also gives examples of central Cantor sets which do not satisfy the assumptions of the Newhouse gap lemma (Theorem \ref{g1}), and still their algebraic difference is equal to $[-1,1]$. We also show that there are central Cantor sets $C(a)$, $C(b)$ with Hausdorff dimension equal to $0$ such that $C(a)-C(b) =[-1,1]$. In Section 3. we give a full characterization of the case, when the set $C(a)-C(b)$ is a finite union of intervals along with corollaries concerning middle Cantor sets.

\section{Sufficient and necessary conditions for the algebraic difference of central Cantor sets to be an interval}
Let us recall the construction of a central Cantor subset of $[0,1]$ (see e.g. \cite{BFN}).

An interval $I$ is called concentric with 
an interval $J$ if they have a common centre.

Let $a = (a_n)$ be a sequence such that $a_n \in (0,1)$ for any $n \in 
\mathbb{N}$ and $I := [0, 1]$. In the first step of the construction, we remove from $I$ the
open interval $P$ centred at $\frac{1}{2}$ of length $a_1$. Then by $I_0$ and $I_1$ we denote, respectively, the left and the right components of $I\setminus P$ (each of length $d_1=\frac{1-a_1}{2}$). Generally, assume that for some $n \in \N$ and $t_1, t_2, \dots, t_n \in \{0,1\}$ we have constructed 
the interval $I_{t_1,\dots , t_n}$ of length $d_{n}$. Denote by $P_{t_1, \dots, t_n}$
 the open interval of length $a_{n+1} d_{n}$, concentric with $I_{t_1,\dots , t_n}$. 
 Now, let $I_{t_1, \dots, t_n, 0}$ and $I_{t_1, \dots , t_n, 1}$
be, respectively, the left and the right components of the set 
$I_{t_1, \dots , t_n} \setminus P_{t_1, \dots , t_n}$. 
By $d_{n+1}$ denote the common length
of both components. 

For every $n \in \mathbb{N}$, denote 
\begin{equation*}
\mathcal{I}_n:=\{I_{t_1, \dots ,t_n}\colon (t_1,\dots
,t_n)\in\{0,1\}^n\}\quad\mbox{ and } \quad C_n(a):= \bigcup \mathcal{I}_n.
\end{equation*}
Let $C(a): = \bigcap_{n\in\mathbb{N}}C_n(a)$. Then $C(a)$ is called a 
central Cantor set. 

Define the thickness of a Cantor set $C$ by (see \cite{Ta})
\begin{equation*}
\tau (C)=\inf_{G_{1}<G_{2}}\max \left\{\frac{l(G_{2})-r(G_{1})}{|G_{1}|}, \;\frac{%
l(G_{2})-r(G_{1})}{|G_{2}|}\right\},
\end{equation*}%
where $G_{1}$, $G_{2}$ are gaps of the set $C$, and $G_{1}<G_{2}$
means that $G_{1}$ is on the left of $G_{2}$.

In the case of central Cantor sets, we have an explicit formula for the thickness, if gaps which appear in any step of the construction are shorter than those which appeared on earlier steps.
\begin{lemma}
\label{l2} \cite{HKY}
If for any $n \in \N$, $a_{n+1} < \frac{2a_n}{1-a_n}$, then
$\tau (C((a_n))) = \inf\limits_{n \in \mathbb{N}} \frac{1-a_n}{2a_n}$.
\end{lemma}
The following theorem is a version for central Cantor sets of a known result about the algebraic difference of two Cantor sets, which uses the notion of thickness. 
\begin{theorem}
[the Newhouse gap lemma; \cite{N},\cite{As}] \label{g1}
If $a,b\in (0,1)^{\mathbb{N}}$ and $\tau
(C(a))\tau (C(b))\geq 1$, then $C(a)-C(b)=[-1,1]$.
\end{theorem}

%
%
%
%
%

Our purpose is to study the algebraic difference $C(a)-C(b)$ of two central Cantor sets. Observe that since $C(b)$ is symmetric with respect to $\frac{1}{2}$ we have $C(a)+C(b) = C(a)+1-C(b)= C(a)-C(b)+1,$ so the algebraic sums and differences of central Cantor sets are topologically the same.  We will use some ideas from \cite{FN}.

Let $t$ and $s$ be some finite sequences. To
denote the concatenation $t$ and $s$, we write $t\symbol{94}%
s$. For $n \in \N$, by $t_n$ we denote the $n$-th term of the sequence $t$ and by $t|n$ we denote the sequence of $n$ first terms of the sequence $t$.

Let $d_{n}=\prod_{i=1}^{n}\frac{1-a_{i}}{2}$ and 
$g_{n}=\prod_{i=1}^{n}\frac{1-b_{i}}{2}$. By $I_{s}^{a}$ ($I_{s}^{b}$, respectively) we will denote the interval $I_{s}$ from the construction of the set $C(a)$ ($%
C(b)$, respectively). Let $\left\{ 0,1\right\} ^{0}=\emptyset$ (the empty sequence), $I_{\emptyset }^{a}=I_{\emptyset }^{b}=\left[
0,1\right] $ and $d_{0}=g_{0}=1$.
Then we have 
\begin{equation*}
C_{n}\left( a\right) -C_{n}\left( b\right) =\bigcup\limits_{p,q\in \left\{
0,1\right\} ^{n}}\left( I_{p}^{a}-I_{q}^{b}\right) .
\end{equation*}%
For $p,q\in \left\{ 0,1\right\} ^{n}$ we define the sequence $s\in \left\{
0,1,2,3\right\} ^{n}$ and the interval $J_{s}$, putting
$s_{i}:=2p_{i}-q_{i}+1$ for $i=1,\ldots n$, and 
$
J_{s}:=I_{p}^{a}-I_{q}^{b}.
$
Then 
\begin{equation*}
C_{n}\left( a\right) -C_{n}\left( b\right) =\bigcup\limits_{s\in \left\{
0,1,2,3\right\} ^{n}}J_{s}
\end{equation*}
and $\left\vert J_{s}\right\vert =d_{n}+g_{n}$ for $s\in \left\{
0,1,2,3\right\} ^{n}$. Observe that
\begin{eqnarray*}
J_{0} &=&I_{0}^{a}-I_{1}^{b}=[-1,-1+d_{1}+g_{1}], \\
J_{1} &=&I_{0}^{a}-I_{0}^{b}=[-g_{1},d_{1}], \\
J_{2} &=&I_{1}^{a}-I_{1}^{b}=[-d_{1},g_{1}], \\
J_{3} &=&I_{1}^{a}-I_{0}^{b}=[1-d_{1}-g_{1},1].
\end{eqnarray*}%
Moreover, if for some $n\in \mathbb{N}$ and $s\in \left\{
0,1,2,3\right\} ^{n}$ we have $J_{s}=I_{p}^{a}-I_{q}^{b}$, then
\begin{eqnarray*}
J_{s\symbol{94}0} &=&I_{p\symbol{94}0}^{a}-I_{q\symbol{94}1}^{b}=[
l( I_{p}^{a}) ,l( I_{p}^{a}) +d_{n+1}] -[
r( I_{q}^{b}) -g_{n+1},r( I_{q}^{b}) ] \\
&=&[l(J_{s}),l(J_{s})+d_{n+1}+g_{n+1}], \\
J_{s\symbol{94}1} &=&I_{p\symbol{94}0}^{a}-I_{q\symbol{94}0}^{b}=[
l( I_{p}^{a}) ,l( I_{p}^{a}) +d_{n+1}] -[
l( I_{q}^{b}) ,l( I_{q}^{b}) +g_{n+1}] \\
&=&[l(J_{s})+g_{n}-g_{n+1},r(J_{s})-d_{n}+d_{n+1}], \\
J_{s\symbol{94}2} &=&I_{p\symbol{94}1}^{a}-I_{q\symbol{94}1}^{b}=[
r( I_{p}^{a}) -d_{n+1},r( I_{p}^{a}) ] -[
r( I_{q}^{b}) -g_{n+1},r( I_{q}^{b}) ] \\
&=&[l(J_{s})+d_{n}-d_{n+1},r(J_{s})-g_{n}+g_{n+1}], \\
J_{s\symbol{94}3} &=&I_{p\symbol{94}1}^{a}-I_{q\symbol{94}0}^{b}=[
r( I_{p}^{a}) -d_{n+1},r( I_{p}^{a}) ] -[
l( I_{q}^{b}) ,l( I_{q}^{b}) +g_{n+1}] \\
&=&[r(J_{s})-d_{n+1}-g_{n+1},r(J_{s})].
\end{eqnarray*}%
Put $J_{\emptyset }=I_{\emptyset }^{a}-I_{\emptyset }^{b}=%
\left[ -1,1\right] $ and observe that the above formulas remain true for 
$n=0$ and $s=\emptyset $.
\begin{lemma}
\label{lem1} For any $a\in (0,1)^{\mathbb{N}}$, $n, k \in \N \cup \{0\},$ where $n>k$, we have $%
d_{n}-d_{n+1}<d_{k}-d_{k+1}$.
\end{lemma}
\begin{proof}
It suffices to show that $%
d_{n}-d_{n+1}<d_{n-1}-d_{n} $, for $n > 1$, or equivalently $2d_{n}-d_{n+1}<d_{n-1}$. Dividing
both sides of the last inequality by $d_{n-1}$, we get $$%
1-a_{n}-\frac{(1-a_{n})(1-a_{n+1})}{4}<1,$$ which holds for all $n$.
\end{proof}
\begin{lemma}
\label{l3}Assume that $a=(a_{n})\in (0,1)^{\mathbb{N}%
},b=(b_{n})\in (0,1)^{\mathbb{N}}$, $n\in \mathbb{N\cup }\left\{ 0\right\} $
and $s\in \left\{ 0,1,2,3\right\} ^{n}$. The following equivalences hold:

\begin{enumerate}
\item $\frac{g_{n}}{d_{n}}\geq a_{n+1}\Leftrightarrow l\left( J_{s\symbol{94}%
2}\right) \leq r\left( J_{s\symbol{94}1}\right) ;$

\item $\frac{d_{n}}{g_{n}}\geq b_{n+1}\Leftrightarrow l\left( J_{s\symbol{94}%
1}\right) \leq r\left( J_{s\symbol{94}2}\right) ;$

\item $\left( \frac{d_{n}}{g_{n}}\geq b_{n+1} \;\text{ and }\; \frac{g_{n}}{d_{n}}\geq
a_{n+1}\right) \Leftrightarrow J_{s\symbol{94}1}\cap J_{s\symbol{94}2}\neq
\emptyset ;$

\item $\frac{d_{n+1}}{g_{n}}\geq b_{n+1}\Leftrightarrow J_{s%
\symbol{94}0}\cap J_{s\symbol{94}1}\neq \emptyset \Leftrightarrow J_{s%
\symbol{94}2}\cap J_{s\symbol{94}3}\neq \emptyset ;$

\item $\frac{g_{n+1}}{d_{n}}\geq a_{n+1}\Leftrightarrow J_{s%
\symbol{94}0}\cap J_{s\symbol{94}2}\neq \emptyset \Leftrightarrow J_{s%
\symbol{94}1}\cap J_{s\symbol{94}3}\neq \emptyset .$
\end{enumerate}

\begin{proof}
Ad (1) We have
\begin{gather*}
r\left( J_{s\symbol{94}1}\right) - l\left( J_{s\symbol{94}2}\right) = r(J_{s})-d_{n}+d_{n+1} - l(J_{s})-d_{n}+d_{n+1}= d_n +g_n -2d_n +2d_{n+1} = g_n-d_n + 2d_{n+1}.
\end{gather*}
Hence $l\left( J_{s\symbol{94}2}\right) \leq r\left( J_{s\symbol{94}1}\right)$ if and only if $g_n-d_n + 2d_{n+1} \geq 0$, which is equivalent to $\frac{g_{n}}{d_{n}}\geq a_{n+1}$.

Ad (2) The proof is analogous to that of (1).

Ad (3) The assertion follows from (1), (2) and the equivalence
\begin{equation*}
\left( l\left( J_{s\symbol{94}2}\right) \leq r\left( J_{s\symbol{94}%
1}\right) \;\text{ and }\; l\left( J_{s\symbol{94}1}\right) \leq r\left( J_{s\symbol{94%
}2}\right) \right) \Leftrightarrow J_{s\symbol{94}1}\cap J_{s\symbol{94}%
2}\neq \emptyset .
\end{equation*}

Ad (4) Of course, $J_{s%
\symbol{94}0}\cap J_{s\symbol{94}1}\neq \emptyset$ if and only if $r(J_{s%
\symbol{94}0}) \geq l(J_{s%
\symbol{94}1})$.
We have 
\begin{gather*}
r(J_{s\symbol{94}0}) -l(J_{s\symbol{94}1}) =  l(J_{s})+d_{n+1}+g_{n+1}-
l(J_{s})-g_{n}+g_{n+1} = d_{n+1}+2g_{n+1}-g_n.
\end{gather*}
Thus, $J_{s%
\symbol{94}0}\cap J_{s\symbol{94}1}\neq \emptyset$ if and only if $d_{n+1}+2g_{n+1}-g_n \geq 0$, which is equivalent to $\frac{d_{n+1}}{g_{n}}\geq b_{n+1}$.

The equivalence $J_{s\symbol{94}0}\cap J_{s\symbol{94}1}\neq
\emptyset \Leftrightarrow J_{s\symbol{94}2}\cap J_{s\symbol{94}3}\neq
\emptyset $ follows from the equality
\begin{equation*}
r\left( J_{s\symbol{94}0}\right) -l\left( J_{s\symbol{94}1}\right)
=d_{n+1}+2g_{n+1}-g_{n}=r\left( J_{s\symbol{94}2}\right) -l\left( J_{s%
\symbol{94}3}\right) .
\end{equation*}

Ad (5) The proof is similar to that of (4).
\end{proof}
\end{lemma}

\begin{lemma}
\label{l4}Assume that $a=(a_{n})\in (0,1)^{\mathbb{N}%
},b=(b_{n})\in (0,1)^{\mathbb{N}}$ and $n\in \mathbb{N}\cup \left\{
0\right\} $. If
\begin{equation*}
\frac{d_{n}}{g_{n}}\geq b_{n+1} \;\text{ and }\; \frac{g_{n}}{d_{n}}\geq a_{n+1}, \;\text{ and }\;
\left( \frac{d_{n+1}}{g_{n}}\geq b_{n+1} \;\text{ or }\; \frac{g_{n}}{%
d_{n}}\geq a_{n+1}\right) ,
\end{equation*}%
then $C_{n+1}\left( a\right) -C_{n+1}\left( b\right) =C_{n}\left( a\right)
-C_{n}\left( b\right) $.

\begin{proof}
Let $s \in \left\{ 0,1,2,3\right\} ^{n}$.
From Lemma \ref{l3} we infer that 
\begin{equation*}
J_{s\symbol{94}1}\cap J_{s\symbol{94}2}\neq \emptyset \;\text{ and }\; J_{s\symbol{94}%
0}\cap J_{s\symbol{94}1}\neq \emptyset \;\text{ and }\; J_{s\symbol{94}2}\cap J_{s%
\symbol{94}3}\neq \emptyset 
\end{equation*}%
or 
\begin{equation*}
J_{s\symbol{94}1}\cap J_{s\symbol{94}2}\neq \emptyset \;\text{ and }\; J_{s\symbol{94}%
0}\cap J_{s\symbol{94}2}\neq \emptyset \;\text{ and }\; J_{s\symbol{94}1}\cap J_{s%
\symbol{94}3}\neq \emptyset .
\end{equation*}%
In both cases we have $J_{s\symbol{94}0}\cup J_{s\symbol{94}1}\cup
J_{s\symbol{94}2}\cup J_{s\symbol{94}3}=J_{s}$. Hence
\begin{eqnarray*}
C_{n}\left( a\right) -C_{n}\left( b\right)  &=&\bigcup\limits_{s\in \left\{
0,1,2,3\right\} ^{n}}J_{s}=\bigcup\limits_{s\in \left\{ 0,1,2,3\right\}
^{n}}\left( J_{s\symbol{94}0}\cup J_{s\symbol{94}1}\cup J_{s\symbol{94}%
2}\cup J_{s\symbol{94}3}\right)  \\
&=&\bigcup\limits_{t\in \left\{ 0,1,2,3\right\} ^{n+1}}J_{t}=C_{n+1}\left(
a\right) -C_{n+1}\left( b\right) .
\end{eqnarray*}
\end{proof}
\end{lemma}

We need one more useful lemma. The proof can be found in \cite{FN} (Proposition 1.1. (9)).
\begin{lemma}
For any nonincreasing sequences $(A_n)$ and $(B_n)$ of compact subsets of $\R$ we have
$$\bigcap_{n\in\N} A_n - \bigcap_{n\in\N} B_n = \bigcap_{n\in\N}(A_n-B_n).$$
In particular, for any $a,b \in (0,1)^{\N}$ we have 
$$C(a) - C(b) = \bigcap_{n\in\N}(C_n(a)-C_n(b)).$$

\end{lemma}
Now, we can prove the main theorem of this section. 

\begin{theorem}
\label{tw1} Let $a=(a_{n})\in (0,1)^{\mathbb{N}},b=(b_{n})\in (0,1)^{%
\mathbb{N}}$.

\begin{itemize}
\item[(1)] If for any $n\in \mathbb{N\cup }\left\{ 0\right\} $
\begin{equation*}
(\ast )\,\,\,\frac{g_{n+1}}{d_{n}}\geq a_{n+1} \;\text{ or }\; \frac{%
d_{n+1}}{g_{n}}\geq b_{n+1}
\end{equation*}%
and
\begin{equation*}
(\ast \ast )\,\,\,\frac{d_{n}}{g_{n}}\geq b_{n+1} \;\text{ and }\; \frac{g_{n}}{d_{n}}%
\geq a_{n+1},
\end{equation*}%
then $C(a)-C(b)=[-1,1]$.

\item[(2)] If conditions $(\ast )$ and $(\ast \ast )$ hold for sufficiently large  $n$, then $C(a)-C(b)$ is a finite union of closed intervals.

\item[(3)] If $C(a)-C(b)=[-1,1]$, then condition $(\ast )$ holds for all $n\in \mathbb{N}\cup \left\{ 0\right\} $.

\item[(4)] If $C(a)-C(b)$ is a finite union of closed intervals, then condition $(\ast )$ holds for sufficiently large $n$.

\end{itemize}
\end{theorem}

\begin{proof}
Ad (1)--(2) Assume that there is $n_0 \in \N$ such that conditions $(\ast )$ and $(\ast \ast )$ hold for all $n\geq n_{0}$. From Lemma \ref{l4} it follows that $$%
C_{n}\left( a\right) -C_{n}\left( b\right) =C_{n_{0}}\left( a\right)
-C_{n_{0}}\left( b\right) $$ for $n\geq n_{0}$, and thus
\begin{equation*}
C\left( a\right) -C\left( b\right) =\bigcap\limits_{n\in \mathbb{N}}\left(
C_{n}\left( a\right) -C_{n}\left( b\right) \right) =C_{n_{0}}\left( a\right)
-C_{n_{0}}\left( b\right) =\bigcup\limits_{s\in \left\{ 0,1,2,3\right\}
^{n_{0}}}J_{s},
\end{equation*}%
so $C\left( a\right) -C\left( b\right) $ is a finite union of closed intervals. If $n_{0}=0$, that is, conditions
$(\ast )$ and $(\ast \ast )$ hold for all $n$, then $C\left( a\right) -C\left( b\right)
=C_{0}\left( a\right) -C_{0}\left( b\right) =\left[ -1,1\right] .$

Ad (3) Assume on the contrary that condition $(\ast )$ does not hold for some $n\in \mathbb{N}\cup \{0\}$. Let $s:=0^{(n)}$. From Lemma \ref{l3} it follows that $J_{s\symbol{94}0}\cap J_{s\symbol{94}1}=\emptyset \ $ and $J_{s\symbol{94}%
0}\cap J_{s\symbol{94}2}=\emptyset .$ Put $L:=\min \{l\left( J_{s%
\symbol{94}1}\right) ,l\left( J_{s\symbol{94}2}\right) \}$. Of course, $%
L>r\left( J_{s\symbol{94}0}\right) =r\left( J_{0^{\left( n+1\right) }}\right)
$. Let $x\in \left( r\left( J_{0^{\left( n+1\right) }}\right) ,L\right) $.
Since $C(a)-C(b)=[-1,1]$, there exists a sequence $u\in
\{0,1,2,3\}^{n+1}$ such that $x\in J_{u}$. Then $u\notin \left\{ s\symbol{%
94}i:i=0,1,2,3\right\} $, so $u_{k}>0$ for some $k\leq n$. In consequence, $x\geq l\left( J_{0^{\left( k-1\right) }\symbol{94}1}\right) $
or $x\geq l\left( J_{0^{\left( k-1\right) }\symbol{94}2}\right) $. Using Lemma \ref{lem1}, in the first case we get
\begin{equation*}
L>x\geq l\left( J_{0^{\left( k-1\right) }\symbol{94}1}\right)
=-1+g_{k-1}-g_{k}>-1+g_{n}-g_{n+1}=l\left( J_{s\symbol{94}1}\right) \geq L,
\end{equation*}%
a contradiction. In the second case, a contradiction is obtained similarly.

Ad (4) Assume on the contrary that $(*)$ does not hold for infinitely many $n \in \N$, and $C(a)-C(b)$ is a finite union of closed intervals. Then there exists $ w > 0$ such that $[-1, -1 +w] \subset C(a) - C(b)$. Let $n \in \N$ be such that condition $(*)$ does not hold and $w > d_n + g_n = r\left( J_{0^{\left( n\right) }}\right) + 1$. Let $s:=0^{(n)}$. Then $J_s \subset [-1,-1+w] \subset C(a)-C(b).$ The rest of the proof is the same as in (3).

\end{proof}

The next example shows that the above theorem gives examples of Cantor sets whose algebraic difference is the interval $[-1,1]$, despite not satisfying assumptions of the Newhouse gap lemma.
\begin{example}
Let $a=(\frac{1}{2},\frac{1}{4},\frac{1}{2},\frac{1}{4}, \dots)$, $b = (\frac{1}{4},\frac{1}{2},\frac{1}{4},\frac{1}{2},\dots)$. Then for $n \in \N \cup \{0\}$ we have
$$\frac{d_{2n}}{g_{2n}} = \frac{(\frac{1}{4}\cdot\frac{3}{8})^n}{(\frac{3}{8}\cdot\frac{1}{4})^n} =1,$$ 
$$\frac{d_{2n+1}}{g_{2n+1}} = \frac{(\frac{1}{4}\cdot\frac{3}{8})^n\cdot\frac{1}{4}}{(\frac{3}{8}\cdot\frac{1}{4})^n\cdot\frac{3}{8}}=\frac{2}{3}.$$ 
Hence for all $n \in \N$ 
$$\frac{g_{2n}}{d_{2n-1}}= \frac{g_{2n-1}}{d_{2n-1}}\cdot\frac{1-b_{2n}}{2} = \frac{3}{2}\cdot\frac{1}{4} = \frac{3}{8} > a_{2n} ,$$
$$\frac{d_{2n-1}}{g_{2n-2}}= \frac{d_{2n-2}}{g_{2n-2}}\cdot\frac{1-a_{2n-1}}{2} = 1\cdot\frac{1}{4} = b_{2n-1} ,$$
$$b_n < \frac{2}{3} \leq \frac{d_{n-1}}{g_{n-1}},$$
$$a_n < 1 \leq \frac{g_{n-1}}{d_{n-1}}.$$
Therefore, conditions $(*)$ and $(**)$ hold for every $n \in \N \cup \{0\}$, so $C(a)-C(b)= [-1,1]$.
Moreover, observe that for any $n \in \N$, $a_{n+1} < \frac{2a_n}{1-a_n}$ and $b_{n+1} < \frac{2b_n}{1-b_n}$. So, by Lemma \ref{l2}, $\tau(C(a))=\tau(C(b)) = \min \{\frac{1}{2},\frac{3}{2}\} = \frac{1}{2}$, and thus $\tau(C(a))\cdot \tau(C(b)) = \frac{1}{4} < 1$, so the sufficient condition from the Newhouse gap lemma does not hold.
\end{example}
Actually, using similar reasoning as in the example above, we can prove the more general result. Prof. Franciszek Prus-Wiśniowski asked the question if for any $\ve > 0$ there exist sequences $a, b \in (0,1)^\N$ such that $\tau(C(a))\cdot \tau(C(b)) \leq \ve$ and $C(a)-C(b)=[-1,1]$. The following proposition provides a positive answer for that question.
\begin{proposition}
Let $\ve >0.$ Then there exist sequences $a, b \in (0,1)^\N$ such that $\tau(C(a))\cdot \tau(C(b)) \leq \ve$ and $C(a)-C(b)=[-1,1]$.
\end{proposition}
\begin{proof}
We will define sequences $a=(a_1,a_2,\dots)$, $b=(b_1,b_2,\dots) \in (0,1)^\N$. Put $a_1:= \frac{1}{3}$, $b_1:=\frac{1}{3}$. Of course, $a_1,b_1 \in (0,1)$. We have $d_0 = g_0 = 1$, so $(**)$ holds for $n=0$. Moreover, 
$$\frac{d_1}{g_0}=d_1 = \frac{1-a_1}{2} = \frac{1}{3} =b_1,$$
thus $(*)$ also holds for $n=0$. 

Put $a_2:= \frac{1}{2\ve+1}$, $b_2:= \frac{\ve}{2\ve+1}$. 
Then, 
$$\frac{g_1}{d_1}=1> a_2$$
and
$$\frac{d_1}{g_1} > \frac{d_2}{g_1} = \frac{1-a_2}{2} = \frac{1-\frac{1}{2\ve+1}}{2} = \frac{\ve}{2\ve+1}=b_2,$$
therefore $(*)$ and $(**)$ hold for $n=1$.
Now, suppose that for some $n \geq 2$ we have defined $a_i,b_i \in (0,1)$, for $i \leq n$ in such a way that conditions $(*)$ and $(**)$ hold for $i \leq n-1$. 
Choose $b_{n+1}\in (0,1)$ such that $b_{n+1} \leq \frac{d_n}{g_n}$ and $b_{n+1} \leq \frac{2b_n}{1-b_n}$. Then, choose $a_{n+1} \in (0,1)$ such that $a_{n+1} \leq \frac{g_{n+1}}{d_n} $ and $a_{n+1} \leq \frac{2a_n}{1-a_n}$. 
Since $\frac{g_{n+1}}{d_n} \leq \frac{g_{n}}{d_n}$, conditions $(*)$ and $(**)$ are satisfied for $n$. 

This way we have inductively constructed sequences $a,b \in (0,1)^\N$ for which conditions $(*)$ and $(**)$ are satisfied for any $n \in \N \cup \{0\}$. Hence $C(a)-C(b) = [-1,1]$. 
In the same time
we have $a_{n+1} \leq \frac{2a_n}{1-a_n}$ and $b_{n+1} \leq \frac{2b_n}{1-b_n}$ for all $n \in \N$ (for $n=1$, $\frac{2a_1}{1-a_1} = \frac{2b_n}{1-a_1} = 1$). Hence
$$\tau(C(a)) = \inf\limits_{n\in\N} \frac{1-a_n}{2a_n} \leq \frac{1-a_2}{2a_2} = \frac{1-\frac{1}{2\ve+1}}{\frac{2}{2\ve+1}} =\ve,$$
$$\tau(C(b)) = \inf\limits_{n\in\N} \frac{1-b_n}{2b_n} \leq \frac{1-b_1}{2b_1} = \frac{1-\frac{1}{3}}{\frac{2}{3}} =1,$$
so
$$\tau(C(a))\cdot \tau(C(b)) \leq \ve.$$

\end{proof}

The characterization of the cases when the set $C(a)-C(a)$ is the interval $[-1,1]$ or a finite union of closed intervals has been already proved with use of various methods (see \cite{AI}, \cite{FN}, \cite{LFGZJ}). However, this result also easily follows from Theorem \ref{tw1}.
\begin{corollary} \label{w1}
Let $a = (a_n) \in (0,1)^{\N}$. Then $C(a) - C(a)$ is equal to:
\begin{itemize}
\item[(1)] the interval $[-1, 1]$ if and only if $a_n \leq \frac{1}{3}$ for all $n \in \N$;

\item[(2)] a finite union of closed intervals if and only if the set $\{n \in \N \colon a_n > \frac{1}{3}\}$ is finite.

\end{itemize}
\end{corollary}
\begin{proof}
For any $n$ we have $\frac{d_n}{g_n} =\frac{d_n}{d_n} = 1$ and $a_n < 1,$ so condition $(**)$ holds for all $n \in \N \cup \{0\}$. Since $\frac{d_n}{d_{n-1}} = \frac{1-a_n}{2}$, the inequality $\frac{d_n}{d_{n-1}} \geq a_n$ is equivalent to $a_{n} \leq \frac{1}{3}.$
From Theorem \ref{tw1} we obtain the assertion.
\end{proof}

In the end of this section let us recall the notion of the Hausdorff dimension (see \cite{F}). For $s>0$, the $s$-dimensional Hausdorff measure of a set 
$E\subset\mathbb{R}$ is defined by the formula 
$H^s(E):=\lim_{\delta\to 0^+}H^s_\delta(E)$ where 
\begin{equation*}
H^s_\delta(E):=\inf\left\{\sum_{i=1}^\infty (\on{diam} I_i)^s\colon
E\subset\bigcup_{i=1}^\infty I_i,\;\on{diam} I_i\leq\delta\right\}.
\end{equation*}
The Hausdorff dimension of $E$ is then given by the formula 
\begin{equation*}
\on{dim}_H(E):=\sup\{s>0\colon H^s(E)>0\}=\inf\{s>0\colon H^s(E)<\infty\}.
\end{equation*}
The determination of the exact Hausdorff measure of various types of Cantor
sets or their sum is an important problem, investigated by many authors (see
e.g. \cite{PP}, \cite{CHW}, \cite{HS} ). For a central Cantor set we have the formula for its Hausdorff dimension, given by Kardos.

\begin{theorem}[\protect\cite{Ka}]
\label{Kar} The Hausdorff dimension of a symmetric Cantor set $C(a)$ is
equal to 
\begin{equation*}
\liminf\limits_{n\to\infty}\frac{n\ln 2}{-\ln d_n}.
\end{equation*}
\end{theorem}

Let $C(a)$ and $C(b)$ be the middle-$\alpha$ and middle-$\beta$ Cantor sets, respectively. In \cite{PS} the authors proved that if $\frac{\ln\frac{1-\alpha}{2}}{\ln \frac{1-\beta}{2}}$ is irrational, then $\on{dim}_H(C(a)+C(b))=\min\{\on{dim}_H(C(a))+\on{dim}_H(C(b)),1\}.$
The following result shows that this is not generally true for arbitrary central Cantor sets.

\begin{proposition}
There are sequences $a,b \in (0,1)^\N$ such that $C(a)-C(b)=[-1,1]$ and 
$\on{dim}_H(C(a))=\on{dim}_H(C(b)) =0$.
\end{proposition}
\begin{proof}
We will define sequences $a$ and $b$ inductively. First, let $a_1 = 1-\frac{2}{2^2}=\frac{1}{2}$ and $b_1 = \frac{1-a_1}{2}$. Then $a_1 < 1 = \frac{g_0}{d_0}$ and 
$b_1 = \frac{d_1}{g_0},$ so conditions $(*)$ and $(**)$ are satisfied for $n=0$. Suppose that for some $n,k \in \N$ we have defined $a_i,b_i$ for $i \leq n$ in such a way that conditions $(*)$ and $(**)$ are satisfied for all $i < n$ and $a_n \geq 1-\frac{2}{ d_{n-1}\cdot 2^{n\cdot 2^k}}$. 
We will define terms $a_{n+i}, b_{n+i}$ for $ i \in \N$, until we can choose $b_{n+j}\geq 1-\frac{2}{g_{n+j-1}\cdot 2^{(n+j)\cdot 2^{k+1}} }$ for some $j \in \N$.   
If $\frac{d_{n}}{g_n} \geq 1$, then put $b_{n+1}=1-\frac{2}{g_{n}\cdot 2^{(n+1)\cdot2^{k+1}} }.$ If $\frac{d_{n}}{g_n} < 1$, then put $b_{n+1} = \frac{d_n}{g_n}$. Choose $a_{n+1} < b_{n+1}$ such that $a_{n+1} \leq \frac{g_{n+1}}{d_n}.$ Observe that $$\frac{d_{n+1}}{g_{n+1}} = \frac{d_n}{g_n}\cdot \frac{1-a_{n+1}}{1-b_{n+1}} > \frac{d_n}{g_n}.$$ 
If $\frac{d_{n+1}}{g_{n+1}} \geq 1$, then put $b_{n+2}=1-\frac{2}{g_{n+1}\cdot2^{(n+2)\cdot2^{k+1}} }.$ If $\frac{d_{n+1}}{g_{n+1}} < 1$, then put $b_{n+2} = \frac{d_{n+1}}{g_{n+1}}$. Choose $a_{n+2} <a_{n+1}< b_{n+1}<b_{n+2}$ such that $a_{n+2} \leq \frac{g_{n+2}}{d_{n+1}}.$ Repeating this procedure, we find $j\in \N$ such that $\frac{d_{n+j}}{g_{n+j}} \geq 1.$ Indeed,
observe that a sequence $\left(\frac{1-a_{n+i}}{1-b_{n+i}} \right)_{i}$ is increasing with terms greater than $1$, thus there is $j \in \N$ such that
$$\frac{d_{n+j}}{g_{n+j}} = \frac{d_n}{g_n} \cdot \frac{1-a_{n+1}}{1-b_{n+1}}\cdot \dots \cdot \frac{1-a_{n+j}}{1-b_{n+j}} \geq 1.$$
Put $b_{n+j+1}=1-\frac{2}{g_{n+j}\cdot2^{(n+j+1)\cdot 2^{k+1}} }$ and $a_{n+j+1} = \frac{g_{n+j+1}}{d_{n+j}}$. 
Since $\frac{d_{n+j}}{g_{n+j}} \geq 1,$ we have $b_{n+j+1} < \frac{d_{n+j}}{g_{n+j}}.$ So, because $b_{n+i} \leq \frac{d_{n+i-1}}{g_{n+i-1}}$ and $a_{n+i} \leq \frac{g_{n+i}}{d_{n+i-1}}$ for $i \leq j+1$, conditions $(*)$ and $(**)$ hold for $0, 1,\dots, n+j$. 

Now, we analogously define $a_i, b_i$ for $i \in \{n+j+2, n+j+3, \dots, m\}$, where $m$ is sufficiently large, in such a way that $\frac{g_{m-1}}{d_{m-1}} \geq 1$, $b_i \leq \frac{d_{i}}{g_{i-1}}$ for $n+j+1 < i \leq m$, $a_i = \frac{g_{i-1}}{d_{i-1}}$ for $n+j+1 < i < m$ and $a_{m} = 1-\frac{2}{d_{m-1}\cdot 2^{m\cdot 2^{k+1}} }$. 

We have defined inductively sequences $a$ and $b \in (0,1)^\N$ such that $(*)$ and $(**)$ hold for $n \in \N \cup \{0\}$ and for every $N,k \in \N$ there are $m,j \geq N$ such that $a_m \geq 1-\frac{2}{d_{m-1}\cdot2^{m\cdot2^k}}$ and $b_j \geq 1-\frac{2}{g_{j-1}\cdot 2^{j\cdot 2^k} }$. 

If $a_n \geq 1-\frac{2}{d_{n-1}\cdot 2^{n\cdot 2^k} }$, then
$$\frac{n\cdot \ln 2}{-\ln d_n} = \frac{n\cdot \ln 2}{\ln \frac{2}{d_{n-1}(1-a_n)}} \leq \frac{n\cdot \ln 2}{\ln 2^{n\cdot 2^k}} = \log_{2^{n\cdot 2^k}} \,2^n = \frac{1}{2^k}.$$
Therefore, for any $k\in \N$
$$\on{dim}_H(C(a)) = \liminf\limits_{n\to \infty}  \frac{n\cdot \ln 2}{-\ln d_n} \leq \frac{1}{2^k},$$
and so $\on{dim}_H(C(a)) = 0.$ 
Similarly,
$\on{dim}_H(C(b)) =0.$

\end{proof}

%
%
\section{Equivalent condition for the algebraic difference of central Cantor sets to be an interval}

Theorem \ref{tw1} from the previous section is useful, but it does not give us an equivalent condition for the algebraic difference of central Cantor sets to be an interval. The condition $(**)$ is not necessary. For instance, it cannot be satisfied for all $n$ if $C(a)$ and $C(b)$ are different middle Cantor sets and we know from the Newhouse gap lemma that their difference can be an interval for some proper constant sequences $a$ and $b$. In this section we will introduce another condition, which will then let us give the characterization of the case, when the algebraic difference of central Cantor sets is an interval. 

First, we introduce some new notation.  Take $a,b \in (0,1)^\N$.
To consider less cases, we will sometimes use symbols $\widehat{1}$ and $\widehat{2}$ instead of $1$ and $2$ in sequences with elements from the set $\{0,1,2,3\}$. If $\widehat{1}$ appears on the $n$-th place of a sequence, then it is equal to $1$ if $g_{n-1}-g_{n} \leq d_{n-1}-d_n$ or it is equal to $2$ otherwise. Then, $\widehat{2} = 3-\widehat{1}.$ That is, $\widehat{1}$ is equal to $1$ and $\widehat{2}=2$ if $l(J_{s\ha 1}) \leq l(J_{s \ha 2})$ for any $s \in \{0,1,2,3\}^{n-1}$. Otherwise, $\widehat{1}=2$ and $\widehat{2}=1$.
We will use a standard arithmetic on the set $\{0,\wh{1},\wh{2},3\}$. In particular, $0+\wh{1} = \wh{1}$, $\wh{1} +\wh{1} = \wh{2}$ and $\wh{2} + \wh{1} = 3.$ Also put
$$L_n:= \min\{d_{n-1}-d_n,g_{n-1}-g_n\}$$
and
$$M_n:=\max\{d_{n-1}-d_n,g_{n-1}-g_n\}.$$
So, $$l(J_{s\ha \widehat{1}}) = l(J_s) + L_n = l(J_{s\ha 0})+L_n$$
and
$$l(J_{s\ha \widehat{2}}) = l(J_s) + M_n.$$
Moreover,
$$l(J_{s\ha 3}) = r(J_s) - d_{n+1}-g_{n+1} = l(J_s) + d_n+g_n - d_{n+1} - g_{n+1} = l(J_s) + L_n + M_n = l(J_{s\ha \wh{2}})+ L_n.$$
 
Now, suppose that for some $n\in \N$ condition $(*)$ is satisfied, but $(**)$ is not. Then, by Lemma \ref{l3}, for any $s \in \{0,1,2,3\}^{n}$ we have $J_{s\ha 1} \cap J_{s \ha 2} = \emptyset$. Also
$J_{s\ha 0} \cap J_{s\ha \widehat{1}} \neq \emptyset$ and $J_{s\ha \widehat{2}} \cap J_{s\ha 3} \neq \emptyset$. This way there appears a gap $(r(J_{s\ha \widehat{1}}), l(J_{s\ha \widehat{2}}))$ in $J_s$. We will denote it by $G_s$. So,
$$G_s = (l(J_s) + L_{n+1} +d_{n+1}+g_{n+1}, l(J_s)+M_{n+1}).$$

%

\begin{proposition}\label{nowy war}
Assume that $a\in (0,1)^{\mathbb{N}%
},b\in (0,1)^{\mathbb{N}}$ and $n\in \mathbb{N}\cup \left\{
0\right\} $. If condition $(*)$ holds for $n$ and
$$
(***)\,\,\, \exists_{m \in \N \cup \{0\}, m\leq n}  \left\{ \begin{array}{ccc}
\forall_{k\in \N \cup \{0\}, k<m}\,\, \sum_{i=0}^{k} L_{n+1-i} +g_{n+1}+d_{n+1} \geq L_{n-k} \\ 
\sum_{i=0}^{m} L_{n+1-i} +g_{n+1}+d_{n+1} \geq M_{n+1}
\end{array},
\right. $$
then $C_{n+1}\left( a\right) -C_{n+1}\left( b\right) =C_{n}\left( a\right)
-C_{n}\left( b\right) $.
\end{proposition}
\begin{proof}
Assume that $(*)$ and $(***)$ hold for $n$.

If $m$ from $(***)$ is equal to $0$, then we have
$$L_{n+1}+g_{n+1}+d_{n+1}\geq M_{n+1},$$
that is, for any $s \in \{0,1,2,3\}$
$$l(J_s)+L_{n+1}+g_{n+1}+d_{n+1}\geq l(J_s)+M_{n+1},$$
so
$$l(J_{s\ha \wh{1}})+g_{n+1}+d_{n+1} \geq l(J_{s\ha \wh{2}}).$$
Hence $$r(J_{s\ha \wh{1}}) \geq l(J_{s\ha \wh{2}}),$$
which means that $J_{s\ha 1} \cap J_{s\ha 2} \neq \emptyset$ and this is equivalent to $(**)$, by Lemma \ref{l3}.
Therefore, we have the assertion.

Now, suppose that $m > 0$. Of course, $C_{n+1}(a)-C_{n+1}(b) \subset C_n(a)-C_n(b) = \bigcup_{s \in \{0,1,2,3\}^n} J_s$.
From Lemma \ref{l3} we infer that for any $s \in \{0,1,2,3\}^n$
\begin{equation*}
J_{s\symbol{94}%
0}\cap J_{s\symbol{94}\wh{1}}\neq \emptyset \;\text{ and }\; J_{s\symbol{94}\wh{2}}\cap J_{s%
\symbol{94}3}\neq \emptyset. 
\end{equation*}%
Hence $$J_s = J_{s\ha 0} \cup  J_{s\ha 1} \cup  J_{s\ha 2} \cup  J_{s\ha 3} \cup G_s.$$ By the definition, 
$J_{s\ha 0} \cup  J_{s\ha 1} \cup  J_{s\ha 2} \cup  J_{s\ha 3} \subset C_{n+1}(a)-C_{n+1}(b)$.
So, to finish the proof, we need to show that $G_s \subset C_{n+1}(a)-C_{n+1}(b)$ for any $s \in \{0,1,2,3\}^n$. 
First, we will inductively prove that for every $k \in \{1,2,\dots, m\}$ we have\begin{equation}\label{ind}
 \left( l(G_{t\ha j_1\ha j_2 \ha \dots \ha j_k}), l(J_{t\ha j_1\ha j_2 \ha \dots \ha j_k}) + \sum_{i=0}^{k} L_{n+1-i} +g_{n+1}+d_{n+1}\right] \subset C_{n+1}(a)-C_{n+1}(b)
\end{equation}
for all $t \in \{0,1,2,3\}^{n-k}$ and $j_1,j_2,\dots, j_k \in \{0,\wh{2}\}$.

By $(***)$, we have
$$L_{n+1}+g_{n+1}+d_{n+1} \geq L_n.$$
Hence for any $t\in\{0,1,2,3\}^{n-1}$
$$l(J_t)+ L_{n+1}+g_{n+1}+d_{n+1} \geq l(J_t) + L_n,$$
so
$$l(G_{t\ha 0}) \geq l(J_{t\ha \wh{1}\ha 0}).$$ 
Also $$l(J_t)+M_n+ L_{n+1}+g_{n+1}+d_{n+1} \geq l(J_t) + L_n+M_n,$$
so
$$l(G_{t\ha \wh{2}}) \geq l(J_{t\ha 3\ha 0}).$$
Since $(*)$ holds for $n$, we know that $[l(J_{t\ha (j+\wh{1})\ha 0}),r(J_{t\ha (j+\wh{1}) \ha \wh{1}}] \subset C_{n+1}(a)-C_{n+1}(b),$
where $j \in \{0,\wh{2}\}.$
Thus, 
$$(l(G_{t\ha j}), l(J_{t\ha j}) + L_n +L_{n+1}+ d_{n+1}+g_{n+1}]=(l(G_{t\ha j}), r(J_{t\ha (j+\wh{1}) \ha \wh{1}}] $$$$\subset [l(J_{t\ha (j+\wh{1})\ha 0}),r(J_{t\ha (j+\wh{1}) \ha \wh{1}}] \subset C_{n+1}(a)-C_{n+1}(b),$$
which proves that (\ref{ind}) is satisfied for $k=1$.

Now, suppose that (\ref{ind}) is satisfied for some $k<m$. 
Then, by $(***)$
$$\sum_{i=0}^{k} L_{n+1-i} +g_{n+1}+d_{n+1} \geq L_{n-k},$$
and therefore for any $t\in\{0,1,2,3\}^{n-k-1}$ and $j_1,j_2, \dots, j_{k+1} \in \{0,\wh{2}\}$ we have
$$ l(J_{t\ha j_1\ha j_2 \ha \dots \ha j_{k+1}}) + \sum_{i=0}^{k} L_{n+1-i} +g_{n+1}+d_{n+1} \geq l(J_{t\ha j_1\ha j_2 \ha \dots \ha j_{k+1}}) + L_{n-k} = l(J_{t \ha (j_1 + \wh{1}) \ha j_2 \dots \ha j_{k+1}}).$$
Since $(*)$ holds for $n$ we have
$$\left[l(J_{t \ha (j_1 + \wh{1}) \ha j_2 \dots \ha j_{k+1}\ha 0}), r( J_{t \ha (j_1 + \wh{1}) \ha j_2 \dots \ha j_{k+1}\ha \wh{1}}) \right] \subset C_{n+1}(a)-C_{n+1}(b).$$
Moreover, since (\ref{ind}) is satisfied for $k$, we have 
$$\left( l(G_{t\ha (j_1+\wh{1})\ha j_2 \ha \dots \ha j_{k+1}}), l(J_{t\ha (j_1+\wh{1})\ha j_2 \ha \dots \ha j_{k+1}}) + \sum_{i=0}^{k} L_{n+1-i} +g_{n+1}+d_{n+1}\right] \subset C_{n+1}(a)-C_{n+1}(b)$$
and
$$\left( l(G_{t\ha j_1\ha j_2 \ha \dots \ha j_{k+1}}), l(J_{t\ha j_1\ha j_2 \ha \dots \ha j_{k+1}}) + \sum_{i=0}^{k} L_{n+1-i} +g_{n+1}+d_{n+1}\right] \subset C_{n+1}(a)-C_{n+1}(b).$$
Because
$$ l(J_{t\ha j_1\ha j_2 \ha \dots \ha j_{k+1}}) + \sum_{i=0}^{k} L_{n+1-i} +g_{n+1}+d_{n+1} \geq  l(J_{t \ha (j_1 + \wh{1}) \ha j_2 \dots \ha j_{k+1}})$$
and
$$l(J_{t\ha (j_1+\wh{1})\ha j_2 \ha \dots \ha j_{k+1}}) = l(J_{t\ha j_1\ha j_2 \ha \dots \ha j_{k+1}}) + L_{n-k},$$
we obtain
$$\left( l(G_{t\ha j_1\ha j_2 \ha \dots \ha j_{k+1}}), l(J_{t\ha j_1\ha j_2 \ha \dots \ha j_{k+1}}) + \sum_{i=0}^{k+1} L_{n+1-i} +g_{n+1}+d_{n+1}\right] \subset C_{n+1}(a)-C_{n+1}(b).$$
By induction, (\ref{ind}) holds for all $k\leq m$. 

By $(***)$, 
$$\sum_{i=0}^{m} L_{n+1-i} +g_{n+1}+d_{n+1} \geq M_{n+1}.$$
Hence for any $t \in \{0,1,2,3\}^{n-m}
$ and all $j_1,j_2,\dots,j_m \in \{0,\wh{2}\}$ we have
$$l(J_{t\ha j_1 \ha \dots \ha j_m}) +\sum_{i=0}^{m} L_{n+1-i} +g_{n+1}+d_{n+1} \geq  l(J_{t\ha j_1 \ha \dots \ha j_m}) + M_{n+1} = r(G_{t\ha j_1 \ha \dots \ha j_m}).$$
Therefore, knowing that (\ref{ind}) holds for $m$, we get
$$G_{t\ha j_1 \ha \dots \ha j_m} \subset \left( l(G_{t\ha j_1\ha j_2 \ha \dots \ha j_{m}}), l(J_{t\ha j_1\ha j_2 \ha \dots \ha j_{m}}) + \sum_{i=0}^{m} L_{n+1-i} +g_{n+1}+d_{n+1}\right] \subset C_{n+1}(a)-C_{n+1}(b).$$

Similarly, we prove that for any
$k \leq m$, $t \in \{0,1,2,3\}^{n-m}$ and all $j_1,j_2,\dots, j_m \in \{\wh{1}, 3\}$
$$\left[r(J_{t\ha j_1\ha j_2 \ha \dots \ha j_k}) - \sum_{i=0}^{k} L_{n+1-i} -g_{n+1}-d_{n+1},r(G_{t\ha j_1\ha j_2 \ha \dots \ha j_k})\right) \subset C_{n+1}(a)-C_{n+1}(b)$$
and thus, using $(***),$ we obtain
$$G_{t\ha j_1 \ha \dots \ha j_m} \subset C_{n+1}(a)-C_{n+1}(b).$$

Now, we will show that for any $s \in \{0,1,2,3\}^n$, $G_s \subset C_{n+1}(a)-C_{n+1}(b).$
For $s \in \{0,1,2,3\}^n$ put $N_s:=\max\{i \leq n\colon s_{i} \in \{\wh{1},3\}\}$ if $s_n \in \{0,\wh{2}\}$ or $N_s:=\max\{i \leq n\colon s_{i} \in \{0,\wh{2}\}\}$ if $s_n \in \{\wh{1},3\}$. If $\{i \leq n\colon s_{i} \in \{\wh{1},3\}\} = \emptyset$ or $\{i \leq n\colon s_{i} \in \{0,\wh{2}\}\} = \emptyset,$ then we put $N_s := 0$.
Using induction with respect to $N_s$, we will prove that if $N_s \geq 0$, then $G_s \subset C_{n+1}(a)-C_{n+1}(n).$
If $N_s \leq n- m$, then we have already proved that $G_s \subset C_{n+1}(a)-C_{n+1}(b)$. Assume that for some $k \geq n-m$ we have proved that if $N_s \leq k$, then $G_s \subset C_{n+1}(a)-C_{n+1}(b).$ Let $s\in \{0,1,2,3\}^n$ be such that $N_s = k+1$. Suppose that $s_n \in \{0,\wh{2}\}$ (the proof when $s_n \in \{\wh{1},3\}$ is similar).
By (\ref{ind}), we have 
$$
\left( l(G_{s}), l(J_{s}) + \sum_{i=0}^{n-k-1} L_{n+1-i} +g_{n+1}+d_{n+1}\right] \subset C_{n+1}(a)-C_{n+1}(b).$$
Observe that
$$l(J_{s}) + \sum_{i=0}^{n-k-1} L_{n+1-i} +g_{n+1}+d_{n+1} = r(J_{t\ha\wh{1}})=l(G_t),$$
where $t=(s|(k+1))\ha (s_{k+2}+\wh{1}) \ha \dots \ha (s_{n}+\wh{1})$.
Since $t_{k+1} = s_{k+1} \in \{\wh{1},3\}$, we have $N_t \leq k$.
Therefore, by induction hypothesis, $G_t \subset C_{n+1}(a)-C_{n+1}(b).$
Hence $$G_s \subset (l(G_s), l(G_t)] \cup G_t \subset C_{n+1}(a) - C_{n+1}(b).$$
By induction, $G_s \subset C_{n+1}(a)-C_{n+1}(b)$ for any $s\in \{0,1,2,3\}^n,$
which finishes the proof.

\end{proof}
\begin{lemma}\label{lem o M1}
Assume that $a\in (0,1)^{\mathbb{N}%
},b\in (0,1)^{\mathbb{N}}$, $n \in \N$. Then there is $m\in \N \cup \{0\}$, $m \leq n$ such that 
\begin{equation}\label{n3}
\sum_{i=0}^m L_{n+1-i} + g_{n+1} + d_{n+1} \geq M_{n+1}.
\end{equation} 
\end{lemma}
\begin{proof}
Without loss of generality assume that $M_{n+1}=g_n-g_{n+1}$. If there is $k \in \{0, 1,\dots, n\}$ such that $L_{n+1-k} = g_{n-k} - g_{n-k+1}$, then, by Lemma \ref{lem1}, $L_{n+1-k}  \geq M_{n+1},$ so (\ref{n2}) holds for $m=k$. If there is no such $k$, then
$L_i = d_{i-1}-d_i$ for all $i \leq n+1$. Therefore,
$$\sum_{i=0}^n L_{n+1-i} +d_{n+1}+g_{n+1} = d_0-d_1+d_1-d_2+\dots+d_n-d_{n+1}+d_{n+1}+g_{n+1}$$$$=d_0+g_{n+1} > 1 > g_{n}-g_{n+1} = M_{n+1},$$
so (\ref{n3}) holds for $m=n$.
\end{proof}
\begin{lemma}\label{lem o M2}
Assume that $a\in (0,1)^{\mathbb{N}%
},b\in (0,1)^{\mathbb{N}}$. For all $n, k \in \N$ such that $n< k$ we have $L_n>L_k$ and $M_n > M_k$. 
\end{lemma}
\begin{proof}
It suffices to show that for all $n \in \N$ we have $L_{n+1} < L_n$ and $M_{n+1} < M_n$. Without loss of generality assume that $L_n = d_{n-1}-d_n$. If $L_{n+1} = d_n-d_{n+1}$, then the assertion follows from Lemma \ref{lem1}. Suppose that $L_{n+1}= g_n - g_{n+1}$. Then we have
$$L_{n+1} = g_{n}-g_{n+1} \leq d_n -d_{n+1} < d_{n-1}-d_n = L_n$$
and
$$M_{n+1} = d_n - d_{n+1} < d_{n-1}-d_n \leq g_{n-1}-g_n = M_n,$$
which finishes the proof.
\end{proof}

Now, we can prove the main theorem of this section.
\begin{theorem}\label{tw2}
Assume that $a\in (0,1)^{\mathbb{N}%
},b\in (0,1)^{\mathbb{N}}$. Then
\begin{itemize}
\item[(1)] $C(a)-C(b) = [-1,1]$ if and only if conditions $(*)$ and $(***)$ hold for all $n \in \N \cup \{0\}$.

\item[(2)] $C(a)-C(b)$ is a finite union of closed intervals if and only if there is $n_0 \geq 0$ such that conditions $(*)$ and $(***)$ hold for all $n \geq n_0$.
\end{itemize}
\end{theorem}
\begin{proof}
Ad (1)-(2) "$\Leftarrow$"

Assume that there is $n_0 \in \N$ such that conditions $(*)$ and $(***)$ hold for all $n\geq n_{0}$. From Proposition \ref{nowy war} it follows that $$%
C_{n}\left( a\right) -C_{n}\left( b\right) =C_{n_{0}}\left( a\right)
-C_{n_{0}}\left( b\right) $$ for $n\geq n_{0}$, and thus
\begin{equation*}
C\left( a\right) -C\left( b\right) =\bigcap\limits_{n\in \mathbb{N}}\left(
C_{n}\left( a\right) -C_{n}\left( b\right) \right) =C_{n_{0}}\left( a\right)
-C_{n_{0}}\left( b\right) =\bigcup\limits_{s\in \left\{ 0,1,2,3\right\}
^{n_{0}}}J_{s},
\end{equation*}%
so $C\left( a\right) -C\left( b\right) $ is a finite union of closed intervals. If $n_{0}=0$, that is, conditions
$(*)$ and $(***)$ hold for all $n$, then $C\left( a\right) -C\left( b\right)
=C_{0}\left( a\right) -C_{0}\left( b\right) =\left[ -1,1\right] .$

Ad (1) "$\Rightarrow$"
Suppose that $C(a)-C(b) = [-1,1]$. By Theorem \ref{tw1}, we know that $(*)$ holds for all $n \in \N \cup \{0\}$. Suppose that $(***)$ does not hold for some $n \in \N \cup \{0\}$. 
Take minimal $m \in \{0,1,\dots, n\}$ such that $$\sum_{i=0}^m L_{n+1-i} + g_{n+1} + d_{n+1} \geq M_{n+1}.$$ We know that it exists from Lemma \ref{lem o M1}. Since $(***)$ does not hold, we have $m>0$ and there is $k\in \N \cup \{0\},$ $k<m$ such that 
$$\sum_{i=0}^k L_{n+1-i} + g_{n+1} + d_{n+1} < L_{n-k}.$$ 
By the minimality of $m$, we also have 
$$\sum_{i=0}^k L_{n+1-i} + g_{n+1} + d_{n+1} < M_{n+1}.$$
Put $K:= \min\{-1+M_{n+1},-1+L_{n-k}\}=\min\{r(G_t), l(J_{0^{(n-k-1)}\ha \wh{1}})\}$.
Take $$x \in \left(-1+\sum_{i=0}^k L_{n+1-i} + g_{n+1} + d_{n+1},K\right).$$
Observe that 
$$l(G_{0^{(n)}}) = -1 + L_{n+1}+d_{n+1}+g_{n+1} \leq -1+\sum_{i=0}^k L_{n+1-i} + g_{n+1} + d_{n+1} <x< K\leq -1 +M_{n+1} = r(G_{0^{(n)}}),$$
so $x \in G_{0^{(n)}}.$
By the assumption, $x \in C_{n+1}(a)-C_{n+1}(b)$, so there is $t \in \{0,1,2,3\}^{n+1}$ such that $x \in J_t$. Since $x \in G_{0^{(n)}}$ and
$$x > -1+\sum_{i=0}^k L_{n+1-i} + g_{n+1} + d_{n+1} = r(J_{0^{(n-k)}\ha \wh{1}^{(k+1)}}),$$
there is either $j\leq n-k$ such that $t_j > 0$ or there is $j \leq n$ such that $t_j = \wh{2}$ or $t_j=3$.  
Observe that the former case is impossible. Indeed, if  $t_j = \wh{2}$ or $t_j = 3$, then, by Lemma \ref{lem o M2}, 
$$x \geq l(J_t) \geq -1 + M_j > -1 +M_{n+1} > x,$$
a contradiction. 
So, there is $j\leq n-k$ such that $t_j > 0$. hence
$$x \geq l(J_t) \geq -1 + L_j \geq -1 + L_{n-k} \geq K > x,$$
a contradiction. 
Therefore, $(***)$ holds for all $n \in \N \cup \{0\}$.

Ad (2) "$\Rightarrow$"
Suppose that $C(a)-C(b)$ is a finite union of intervals. So, there is $w > 0$ such that $[-1, -1+w] \subset C(a)-C(b).$ By Theorem \ref{tw1}, we know that there is $k \geq 0$ such that $(*)$ holds for all $n \geq k$. Suppose that $(***)$ does not hold for infinitely many $n$. Choose $n \in \N$ such that $(***)$ does not hold for $n$ and $J_{0^{(n)}} \subset [-1,-1+w]$. The rest of the proof is identical as in the part Ad (1) "$\Rightarrow$". 
\end{proof}

The next result shows that the condition from Newhouse gap lemma is not only sufficient, but it is also necessary to obtain an interval as an algebraic difference of middle-$\alpha$ and middle-$\beta$ Cantor sets if $\frac{\ln \frac{1-\alpha}{2}}{\ln \frac{1-\beta}{2}}$ is irrational. This fact has been already noticed by Pourbarat in \cite{P}.

\begin{corollary}\label{WnP}
Let $a$ be a sequence with all terms equal to $\alpha$ and let $b$ be a sequence with all terms equal to $\beta$, where $\alpha, \beta \in (0,1)$ are such that $\frac{\ln \frac{1-\alpha}{2}}{\ln \frac{1-\beta}{2}}$ is irrational.
Then the following conditions are equivalent:
\begin{itemize}
\item[(i)] $\beta \leq \frac{1-\alpha}{1+3\alpha}$;

\item[(ii)] $\tau(C(a))\cdot \tau(C(b)) \geq 1$;

\item[(iii)] $C(a)-C(b) = [-1,1]$.
\end{itemize}
\end{corollary}
\begin{proof}
(i) $\Rightarrow$ (ii)
Since obviously $\alpha<\frac{2\alpha}{1-\alpha}$ and $\beta<\frac{2\beta}{1-\beta}$, by Lemma \ref{l2}, we have $\tau(C(a))= \frac{1-\alpha}{2\alpha}$ and $\tau(C(b)) = \frac{1-\beta}{2\beta}.$
Therefore
$$\tau(C(a))\cdot \tau(C(b)) = \frac{1-\alpha}{2\alpha}\cdot  \frac{1-\beta}{2\beta} \geq \frac{1-\alpha}{2\alpha} \cdot \frac{1-\frac{1-\alpha}{1+3\alpha}}{2\cdot\frac{1-\alpha}{1+3\alpha}} = \frac{1-\alpha}{2\alpha} \cdot \frac{4\alpha}{1+3\alpha}\cdot\frac{1+3\alpha}{2\cdot(1-\alpha)} = 1.$$

(ii) $\Rightarrow$ (iii)
It follows directly from Newhouse gap lemma.

(iii) $\Rightarrow$ (i)


Without loss of generality we can assume that $\alpha < \beta$, because $$\beta \leq \frac{1-\alpha}{1+3\alpha} \Leftrightarrow \alpha \leq \frac{1-\beta}{1+3\beta}.$$ 

Now, observe that for $n > 1$ we have $L_n =g_{n-1}-g_n$. Indeed, since $\beta > \alpha$, we have 
$\frac{1-\beta}{1-\alpha} < 1$, and so the sequence
$\left(\frac{g_n}{d_n}\right) = \left(\frac{(1-\beta)^n}{(1-\alpha)^n}\right)$ is decreasing.
Therefore, the sequence
$\left( \frac{g_{n-1}-g_n}{d_{n-1}-d_n} \right) =\left(\frac{g_{n-1}}{d_{n-1}}\cdot\frac{1-\frac{1-\beta}{2}}{1-\frac{1-\alpha}{2}}\right)$ is is also decreasing and
$$g_1-g_2= \frac{1-\beta}{2}\cdot \left(1-\frac{1-\beta}{2}\right) =\frac{1-\beta}{2}\cdot\frac{1+\beta}{2} = \frac{1-\beta^2}{4} < \frac{1-\alpha^2}{4} = d_1-d_2.$$
Hence $L_n = g_{n-1}-g_n$ for $n\geq 2$. 

Now, we will prove\\
{\bf Claim}\\
If for given $n \in \N$ and $k <n-1$
$$\sum_{i=0}^k L_{n+1-i} + g_{n+1}+d_{n+1} \geq L_{n-k},$$
then also
$$\sum_{i=0}^{k-1} L_{n+1-i} + g_{n+1}+d_{n+1} \geq L_{n-k+1}.$$

By the assumption, we have 
$$\sum_{i=0}^{k-1} L_{n+1-i} + g_{n+1}+d_{n+1} \geq L_{n-k} - L_{n+1-k}.$$
Hence it suffices to prove that $$L_{n-k} - L_{n+1-k} \geq L_{n-k+1}.$$
Since $n-k \geq 2$, we have
$$L_{n-k}-2L_{n+1-k} = g_{n-k-1}-g_{n-k}-2g_{n-k}+2g_{n-k+1}.$$
After division of the expression above by $g_{n-k-1}$ we obtain
$$1-3\cdot\frac{1-\beta}{2} + 2\cdot\frac{(1-\beta)^2}{4} =\frac{\beta}{2}+\frac{\beta^2}{2} \geq 0,$$
so $$L_{n-k} - 2L_{n+1-k} \geq 0.$$
Therefore, $$\sum_{i=0}^{k-1} L_{n+1-i} + g_{n+1}+d_{n+1} \geq L_{n-k} - L_{n+1-k},$$ which finishes the proof of Claim.

Now, we will prove that for any $n >1$
\begin{equation}\label{LiM}
L_{n+1}+g_{n+1}+d_{n+1}-g_n-d_n \geq M_{n+1}-M_n.
\end{equation}
Since $n >1$, (\ref{LiM}) is equivalent to
$$g_n-g_{n+1}+g_{n+1}+d_{n+1}-g_n-d_n \geq d_n-d_{n+1}-d_{n-1}+d_n,$$
and so to
$$d_{n-1}-3d_n+2d_{n+1}\geq 0.$$
Dividing the above inequality by $d_{n-1}$ we receive the inequality
$$1-\frac{3}{2}(1-\alpha) + \frac{1}{2}(1-\alpha)^2 \geq 0,$$
which is satisfied, because $(1-\alpha) < 1$.
This proves (\ref{LiM}).

Since $C(a)-C(b) = [-1,1]$, by Theorem \ref{tw2} condition $(***)$ holds for any $n$, that is, there is $m \leq n$ such that 
 $$\left\{ \begin{array}{ccc}
\forall_{k\in \N \cup \{0\}, k<m}\,\, \sum_{i=0}^{k} L_{n+1-i} +g_{n+1}+d_{n+1} \geq L_{n-k} \\ 
\sum_{i=0}^{m} L_{n+1-i} +g_{n+1}+d_{n+1} \geq M_{n+1}
\end{array}.
\right. $$
Since $M_{n+1} \to 0$, there is $K \in \N$ such that we have $\sum_{i=0}^{K-1} L_{K+1-i} +g_{K+1}+d_{K+1} \geq M_{K+1}$, (that is, $m<n$ for $n=K$).
Moreover, by (\ref{LiM}), if 
$$\sum_{i=0}^{m} L_{n+1-i} +g_{n+1}+d_{n+1} \geq M_{n+1},$$
then
$$\sum_{i=0}^{m+1} L_{n+2-i} +g_{n+2}+d_{n+2} = \sum_{i=0}^{m} L_{n+1-i} +g_{n+1}+d_{n+1} +L_{n+2}+g_{n+2}+d_{n+2}-g_{n+1}-d_{n+1} $$$$\geq M_{n+1}+M_{n+2}-M_{n+1}= M_{n+2}.$$
Thus, for $n \geq K$ we may assume that $m < n$.
Moreover, from Claim we infer that if
$$\sum_{i=0}^{m-1} L_{n+1-i}+g_{n+1}+d_{n+1} \geq L_{n-m+1},$$
then 
$$\sum_{i=0}^{k} L_{n+1-i} +g_{n+1}+d_{n+1} \geq L_{n-k}$$
for all $k < m$.
If $k < n$, then $n+1-k \geq 2$, so  we have
$$\sum_{i=0}^{k} L_{n+1-i}+g_{n+1}+d_{n+1} = \sum_{i=0}^{k} (g_{n-i}-g_{n-i+1}) +g_{n+1} + d_{n+1}$$$$= g_n - g_{n+1} + g_{n-1} -g_n + \dots +g_{n-k} -g_{n-k+1} + g_{n+1}+d_{n+1} = g_{n-k}+d_{n+1}.$$
Therefore, condition $(***)$ holds for $n \geq K$ if and only if there is $m < n$ such that 
\begin{equation}\label{nierownosci}
\left\{ \begin{array}{ccc}
g_{n-m+1} +d_{n+1} \geq g_{n-m} - g_{n-m+1} \\ 
g_{n-m} + d_{n+1} \geq d_n - d_{n+1}
\end{array}.
\right.
\end{equation}
We have
$$  g_{n-m+1} +d_{n+1} \geq g_{n-m} - g_{n-m+1} \Leftrightarrow 2\cdot\frac{g_{n-m+1}}{d_n}-\frac{g_{n-m}}{d_n} +\frac{1-\alpha}{2} \geq 0 $$
and
$$2\cdot \frac{g_{n-m+1}}{d_n}-\frac{g_{n-m}}{d_n} +\frac{1-\alpha}{2} = 2\cdot\left(\frac{1-\beta}{1-\alpha}\right)^n \cdot \left( \frac{2}{1-\beta}\right)^{m-1} - \left(\frac{1-\beta}{1-\alpha}\right)^n \cdot \left( \frac{2}{1-\beta}\right)^{m} + \frac{1-\alpha}{2}$$
$$=\left(\frac{1-\beta}{1-\alpha}\right)^n \cdot \left( \frac{2}{1-\beta}\right)^{m}\cdot(1-\beta-1) + \frac{1-\alpha}{2}= -\beta\cdot \left(\frac{1-\beta}{1-\alpha}\right)^n \cdot \left( \frac{2}{1-\beta}\right)^{m} +\frac{1-\alpha}{2}.$$
Thus, the first inequality in (\ref{nierownosci}) holds if and only if
\begin{equation*}
\left( \frac{2}{1-\beta}\right)^{m} \leq \frac{1-\alpha}{2\beta}\cdot \left(\frac{1-\alpha}{1-\beta}\right)^n,
\end{equation*}
which is equivalent to
\begin{equation*}
\left( \frac{2}{1-\beta}\right)^{m-n} \leq \frac{1-\alpha}{2\beta}\cdot \left(\frac{1-\alpha}{2}\right)^n
\end{equation*}
and finally it is equivalent to
\begin{equation}\label{n1}
\left( \frac{2}{1-\beta}\right)^{n-m} \geq \frac{2\beta}{1-\alpha}\cdot \left(\frac{2}{1-\alpha}\right)^n
\end{equation}
We also have
$$g_{n-m} + d_{n+1} \geq d_n - d_{n+1} \Leftrightarrow \frac{g_{n-m}}{d_n} + 2\cdot\frac{d_{n+1}}{d_n}-1 \geq 0$$ and
$$\frac{g_{n-m}}{d_n} + 2\cdot\frac{d_{n+1}}{d_n}-1 = \left(\frac{1-\beta}{1-\alpha}\right)^n \cdot \left( \frac{2}{1-\beta}\right)^{m} +1-\alpha-1.$$
Hence the second inequality in (\ref{nierownosci}) holds if and only if
\begin{equation*}
\left( \frac{2}{1-\beta}\right)^{m} \geq a\cdot \left(\frac{1-\alpha}{1-\beta}\right)^n,
\end{equation*}
which is equivalent to 
\begin{equation}\label{n2}
\left( \frac{2}{1-\beta}\right)^{n-m} \leq \frac{1}{\alpha}\cdot \left(\frac{2}{1-\alpha}\right)^n.
\end{equation}
We will now show that for any $n \geq K$ there is $m < n$ satisfying (\ref{n1}) and (\ref{n2}) only if
\begin{equation}\label{on}
\frac{1-\alpha}{2\alpha\beta} \geq \frac{2}{1-\beta}.
\end{equation}
Suppose that 
$$\frac{1-\alpha}{2\alpha\beta} < \frac{2}{1-\beta}.$$

It is well known that if for some $c,d \in \R$ we know that $\frac{c}{d}$ is irrational, then the set $\{n\cdot_{mod\, d} c\colon n\in \N\}$ is dense in $[0,d)$.
Using this fact for $c=\ln\frac{2}{1-\alpha}$ and $d=\ln\frac{2}{1-\beta}$, we get that the set \\$\{n\cdot_{mod\, \ln\frac{2}{1-\beta}} \frac{2}{1-\alpha}\colon n\in \N\}$ is dense in $\left[0,\ln\frac{2}{1-\beta}\right)$, and thus also the set  $$A:=\{n\cdot_{mod \ln\frac{2}{1-\beta}} \frac{2}{1-\alpha}+ _{mod\, \ln\frac{2}{1-\beta}} \ln\frac{2\beta}{1-\alpha} \colon n\in \N\}$$ is dense in $\left[0,\ln\frac{2}{1-\beta}\right)$. Since $$\frac{1-\alpha}{2\alpha\beta} < \frac{2}{1-\beta},$$
we have $\frac{4\alpha\beta}{(1-\alpha)(1-\beta)} > 1$. 
Because $C(a)-C(b) = [-1,1]$, we know by Theorem \ref{tw1} that $(*)$ is satisfied for $n = 0$, that is, $\alpha = a_1 \leq \frac{g_1}{d_0} =\frac{1-\beta}{2}$ or $\beta = b_1 \leq \frac{d_1}{g_0} =\frac{1-\alpha}{2}$. In the first case we have
$$\frac{4\alpha\beta}{(1-\alpha)(1-\beta)} \leq \frac{2\beta}{1-\alpha} < \frac{2}{1-\beta},$$
and in the second case,
$$\frac{4\alpha\beta}{(1-\alpha)(1-\beta)} \leq\frac{2\alpha}{1-\beta}< \frac{2}{1-\beta}.$$ 

Therefore in both cases, $\frac{4\alpha\beta}{(1-\alpha)(1-\beta)} \leq \frac{2}{1-\beta}$. Therefore, $\ln\frac{4\alpha\beta}{(1-\alpha)(1-\beta)} \in (0,\ln \frac{2}{1-\beta})$. Using the density of the set $A$, we find $n,k \in \N$ (we can choose $n \geq K$) such that 
$$n\cdot \ln\frac{2}{1-\alpha} +\ln\frac{2\beta}{1-\alpha} \in \left(k\cdot \ln\frac{2}{1-\beta},k\cdot \ln\frac{2}{1-\beta}+ \ln\frac{4\alpha\beta}{(1-\alpha)(1-\beta)}\right).$$
Then $$\frac{2\beta}{1-\alpha}\cdot\left( \frac{2}{1-\alpha}\right)^n \in \left(\left(\frac{2}{1-\beta}\right)^k,\left(\frac{2}{1-\beta}\right)^k \cdot  \frac{4\alpha\beta}{(1-\alpha)(1-\beta)}\right).$$
Therefore,
$$\left(\frac{2}{1-\beta}\right)^{k} < \frac{2\beta}{1-\alpha}\cdot\left( \frac{2}{1-\alpha}\right)^n < \left(\frac{2}{1-\beta}\right)^{k+1},$$ but
$$\frac{1}{\alpha}\cdot\left( \frac{2}{1-\alpha}\right)^n < \frac{1}{\alpha}\cdot\frac{1-\alpha}{2\beta}\cdot \frac{4\alpha\beta}{(1-\alpha)(1-\beta)} \cdot \left(\frac{2}{1-\beta}\right)^{k} = \left(\frac{2}{1-\beta}\right)^{k+1}.$$
Thus, for $n$ there is no $m<n$ such that both inequalities (\ref{n1}) and (\ref{n2}) hold, a contradiction. 
Therefore, (\ref{on}) is satisfied.
Hence 
$$(1-\alpha)(1-\beta)\geq 4\alpha\beta,$$
and finally
$$\beta \leq \frac{1-\alpha}{1+3\alpha}.$$

\end{proof}
From Theorem \ref{tw2} we can also infer the equivalent condition to obtain an interval as an algebraic difference of middle-$\alpha$ and middle-$\beta$ Cantor sets if $\frac{\ln \frac{1-\alpha}{2}}{\ln \frac{1-\beta}{2}}$ is rational. This result has been also already proved by Pourbarat in \cite{P}. 

\begin{corollary} \label{wariq}
Let $a$ be a sequence with all terms equal to $\alpha$ and let $b$ be a sequence with all terms equal to $\beta$, where $\alpha, \beta \in (0,1)$ are such that $\frac{\ln \frac{1-\alpha}{2}}{\ln \frac{1-\beta}{2}} = \frac{n_0}{m_0},$ where $n_0,m_0$ are relatively prime and $n_0 \leq m_0$.
Then $C(a)-C(b) = [-1,1]$ if and only if there is $j \in \{0,1,\dots, m_0-1\}$ such that
\begin{equation}\label{warprz}
\left(\frac{1-\alpha}{2}\right)^\frac{j}{m_0} \in \left[\frac{2\beta}{1-\alpha}\cdot\left(\frac{1-\alpha}{2}\right)^\frac{1}{m_0},\frac{1-\beta}{2\alpha}\right].
\end{equation}
\end{corollary}
\begin{proof}
In the proof of Corollary \ref{WnP} we have already proved that $(***)$ holds for $n \geq K$, where $K$ is as in that proof, if and only if there is $m < n$ such that inequalities (\ref{n1}) and (\ref{n2}) are satisfied. Therefore, to finish the proof we need to show that\begin{itemize}
\item[a)] for every $n \geq K$ there is $m < n$ such that inequalities (\ref{n1}) and (\ref{n2}) hold if and only if there is $j< m_0$ for which (\ref{warprz}) is satisfied;

\item[b)] (\ref{warprz}) implies $(*)$ for all $n\in \N \cup \{0\}$;

\item[c)] (\ref{warprz}) implies that $(***)$ holds for all $n<K$.
\end{itemize} 

Ad a) By the assumption, $m_0\ln\frac{1-\alpha}{2} = n_0 \ln \frac{1-\beta}{2}$, and so,
$$n_0\ln\frac{2}{1-\beta} = m_0 \ln \frac{2}{1-\alpha}.$$
Since $n_0 \leq m_0,$ we have $\alpha \leq \beta$. Moreover, for any $n \in \N \cup \{0\}$ we have 
$$n \cdot \ln \frac{2}{1-\alpha} =_{mod \, \ln\frac{2}{1-\beta}} \frac{k}{m_0} \ln \frac{2}{1-\alpha} $$
for some $k \in \{0,1,\dots m_0-1\}$. Thus, for every $n\in \{0,1,\dots m_0-1\}$ there is $k \in \{0,1,\dots m_0-1\}$ such that
$$\ln\frac{2\beta}{1-\alpha}+_{mod \, \ln\frac{2}{1-\beta}} n\ln \frac{2}{1-\alpha} \in \left[\frac{k}{m_0}\cdot\ln\frac{2}{1-\alpha} ,\frac{k+1}{m_0}\cdot\ln\frac{2}{1-\alpha}\right).$$
Observe that (\ref{warprz}) is satisfied if and only if 
$$\left(\frac{2}{1-\alpha}\right)^\frac{j}{m_0} \in \left[\frac{2\alpha}{1-\beta}, \frac{1-\alpha}{2\beta}\cdot\left(\frac{2}{1-\alpha}\right)^\frac{1}{m_0}\right],$$
and so
$$\frac{2\beta}{1-\alpha}\cdot\left(\frac{2}{1-\alpha}\right)^\frac{j}{m_0} \in \left[\frac{4\alpha\beta}{(1-\beta)(1-\alpha)}, \left(\frac{2}{1-\alpha}\right)^\frac{1}{m_0}\right],$$
and this equivalent to
$$\ln\frac{2\beta}{1-\alpha}+\frac{j}{m_0}\ln\frac{2}{1-\alpha}\in \left[\ln\frac{4\alpha\beta}{(1-\beta)(1-\alpha)}, \frac{1}{m_0}\ln\frac{2}{1-\alpha}\right]$$
and finally to
$$\ln\frac{2\beta}{1-\alpha}+_{mod \,\ln\frac{2}{1-\beta}}k\cdot\ln\frac{2}{1-\alpha}\in \left[\ln\frac{4\alpha\beta}{(1-\beta)(1-\alpha)}, \frac{1}{m_0}\ln\frac{2}{1-\alpha}\right]$$
for some $k \in \{0,1,\dots,m_0-1\}.$
If $\beta \leq \frac{1-\alpha}{1+3\beta}$, then, by the same argument as in the proof of Corollary \ref{WnP}, $C(a)-C(b) = [-1,1]$. So, suppose that $\beta >\frac{1-\alpha}{1+3\beta}$.
Then $\frac{4\alpha\beta}{(1-\alpha)(1-\beta)}>1$, and so $\ln\frac{4\alpha\beta}{(1-\beta)(1-\alpha)} > 0$. 
Hence 
$\left[\ln\frac{4\alpha\beta}{(1-\beta)(1-\alpha)}, \frac{1}{m_0}\ln\frac{2}{1-\alpha}\right]\subset \left[0,\frac{1}{m_0}\ln\frac{2}{1-\alpha}\right].$
Of course, there is only one $k \in \{0,\dots,m_0-1\}$ such that
$$\ln\frac{2\beta}{1-\alpha}+_{mod \,\ln\frac{2}{1-\beta}}k\cdot \ln\frac{2}{1-\alpha}\in \left[\ln\frac{4\alpha\beta}{(1-\beta)(1-\alpha)}, \frac{1}{m_0}\ln\frac{2}{1-\alpha}\right),$$
so if (\ref{warprz}) is satisfied, then for any $n \in \{0,\dots,m_0-1\}$ we have 
$$\ln\frac{2\beta}{1-\alpha}+_{mod \,\ln\frac{2}{1-\beta}}n\cdot\ln\frac{2}{1-\alpha} \geq \ln \frac{4\alpha\beta}{(1-\beta)(1-\alpha)},$$
and thus for any $n \in \N\cup\{0\}$ there is $k \in \N\cup \{0\}$ such that
$$\frac{4\alpha\beta}{(1-\beta)(1-\alpha)}\cdot\left(\frac{2}{1-\beta}\right)^k \leq \frac{2\beta}{1-\alpha}\cdot\left(\frac{2}{1-\alpha}\right)^n \leq \left(\frac{2}{1-\beta}\right)^{k+1}$$
and
$$\frac{1}{\alpha} \cdot\left(\frac{2}{1-\alpha}\right)^n =\frac{1-\alpha}{2\alpha\beta}\cdot \frac{2\beta}{1-\alpha}\left(\frac{2}{1-\alpha}\right)^n\geq \frac{1-\alpha}{2\alpha\beta} \cdot \frac{4\alpha\beta}{(1-\beta)(1-\alpha)}\cdot\left(\frac{2}{1-\beta}\right)^k = \left(\frac{2}{1-\beta}\right)^{k+1}. $$
Therefore, for every $n \geq K$ there is $m < n$ such that inequalities (\ref{n1}) and (\ref{n2}) hold. 

Now, suppose that (\ref{warprz}) does not hold, but $C(a)-C(b) = [-1,1]$. Then condition $(*)$ holds for $n=0$, so $\alpha = a_1 \leq \frac{g_1}{d_0} =\frac{1-\beta}{2}$ or $\beta = b_1 \leq \frac{d_1}{g_0} =\frac{1-\alpha}{2}$. In both cases we obtain
$$\frac{4\alpha\beta}{(1-\alpha)(1-\beta)} < \frac{2}{1-\beta}.$$ 
So, there is $j \in \{0,\dots, m_0-1\}$ such that 
$$\ln\frac{2\beta}{1-\alpha}+_{mod \,\ln\frac{2}{1-\beta}}j\cdot\ln\frac{2}{1-\alpha}\in \left[0, \ln\frac{4\alpha\beta}{(1-\beta)(1-\alpha)}\right).$$ 
Hence for some $k \in \N \cup \{0\}$
$$\left(\frac{2}{1-\beta}\right)^k \leq \frac{2\beta}{1-\alpha}\cdot\left(\frac{2}{1-\alpha}\right)^j <\frac{4\alpha\beta}{(1-\beta)(1-\alpha)}\left(\frac{2}{1-\beta}\right)^k < \left(\frac{2}{1-\beta}\right)^{k+1}$$
and
$$\frac{1}{\alpha} \cdot\left(\frac{2}{1-\alpha}\right)^j \leq \frac{1-\alpha}{2\alpha\beta} \cdot \frac{4\alpha\beta}{(1-\beta)(1-\alpha)}\cdot\left(\frac{2}{1-\beta}\right)^k < \left(\frac{2}{1-\beta}\right)^{k+1}. $$
Therefore, there for all $n=j+i\cdot m_0,$ where $i \in \N \cup\{0\}$ there is no $m < n$ such that inequalities (\ref{n1}) and (\ref{n2}) hold, which finishes the proof of a).

Ad b) If $\tau(C(a))\cdot \tau(C(b)) \geq 1,$ then $(*)$ must be satisfied for all $n$, because $C(a)-C(b)=[-1,1]$. So, suppose that 
 $\tau(C(a))\cdot \tau(C(b)) < 1,$
that is,
$$\frac{(1-\alpha)(1-\beta)}{4\alpha\beta} < 1.$$
Hence
$$1-\alpha-\beta<3\alpha\beta,$$
and so
$$\alpha > \frac{1-\beta}{1+3\beta} .$$ 
If $\beta \leq \frac{1}{3}$, then $\tau(C(b))=\frac{1-\beta}{2\beta} \geq 1$. Because $\alpha < \beta$, then also $\tau(C(a)) \geq 1$, and so $\tau(C(a))\cdot \tau(C(b)) \geq 1,$ a contradiction. Hence $\beta > \frac{1}{3}$.
Therefore \begin{equation}\label{alfa}
\alpha > \frac{1-\beta}{1+3\beta} > \frac{1-\beta}{2} .
\end{equation}
Denote $I = \left[ \frac{2\beta}{1-\alpha}\cdot \left(\frac{1-\alpha}{2}\right)^\frac{1}{m_0}, \frac{1-\beta}{2\alpha}\right]$.
Using again the assumption that $\tau(C(a))\cdot \tau(C(b)) < 1,$ we obtain that $\frac{r(I)}{l(I)} < \left(\frac{2}{1-\alpha}\right)^\frac{1}{m_0}$, so there is only one $j$ satisfying (\ref{warprz}). By (\ref{alfa}) $\frac{1-\beta}{2\alpha}<1,$ so $j$ for which (\ref{warprz}) holds is greater than $0.$ We have 
$$\frac{2\beta}{1-\alpha}\cdot \left(\frac{1-\alpha}{2}\right)^\frac{1}{m_0} \leq \left(\frac{1-\alpha}{2}\right)^\frac{j}{m_0}, $$
and so
$$\frac{2\beta}{1-\alpha} \leq \left(\frac{1-\alpha}{2}\right)^\frac{n-1}{m_0} \leq 1.$$
Therefore, using the fact that $\alpha \leq \beta$, for any $n \in \N \cup \{0\}$ we have 
$$b_{n+1} = \beta \leq \frac{1-\alpha}{2}  \leq \left(\frac{1-\alpha}{1-\beta}\right)^n \cdot \frac{1-\alpha}{2} = \frac{d_{n+1}}{g_n},$$
so $(*)$ is satisfied for all $n \in \N \cup \{0\}.$

Ad c)
For $n=0$, $(***)$ holds always, because
$$L_1+g_1+d_1=d_0-d_1+g_1+d_1=1+g_1\geq 1 > M_1 = g_0-g_1=1-g_1.$$
By (\ref{warprz}), we have
$$\frac{1-\alpha}{2} \leq \frac{1-\beta}{2\alpha}.$$
Hence 
$$\frac{1-\beta}{1-\alpha}+1 -\alpha \geq 1,$$
and so
$$\frac{g_1}{d_1}+\frac{2d_2}{d_1} \geq 1,$$
and thus
$$g_1-g_2+g_2+d_2\geq d_1-d_2.$$
Therefore,
$$L_2+g_2+d_2\geq M_2,$$
so $(***)$ holds for $n=1$, with $m=0$. In particular, $m<n$ for $n=1$, so we can put $K=1$, which finishes the proof.

\end{proof}

There are still a lot of open questions left regarding the algebraic difference of central Cantor sets, but now we know exactly, when it is a finite union of intervals. Although there are some sufficient conditions for the algebraic difference of central Cantor sets to be a Cantor set or a Cantorval, it is far from the full characterization. Even in the case of middle Cantor sets there is still "a mysterious region" introduced by Solomyak in \cite{So}. From Corollaries \ref{WnP} and \ref{wariq} we now exactly know the area, where the difference of middle Cantor sets is an interval (a nice picture can be found in \cite{P}). However, most of the "mysterious region" from \cite{S} is still mysterious. 
It seems possible to use similar geometrical methods as in Theorem \ref{tw2} to find new conditions for the algebraic difference of central Cantor sets to be a Cantor set or a Cantorval. However, there appear some new difficulties which have to be considered.
The first step could be proving that the following statement holds.

\begin{conjecture}
If condition $(***)$ does not hold for almost all $n$, then $C(a)-C(b)$ is a Cantor set.
\end{conjecture}

If the above Conjecture is true, then it would be possible to easily find new examples of Cantor sets which algebraic difference is a Cantor set. For example, it would imply that the set $C(\frac{1}{2}) + C(\frac{1}{3})$ is a Cantor set.
\section*{Acknowledgements} The author would like to thank Tomasz Filipczak for many fruitful conversations during the preparation of the paper. The author would like also to express his gratitude to the anonymous referee for valuable remarks and suggestions which helped to develop the paper.

The author was supported by the GA \v{C}R project 20-22230L.


\end{document}